\documentclass[11pt]{amsart}
\usepackage{amsmath,amsfonts,amsthm,mathrsfs}
\usepackage{amssymb}

\usepackage[unicode,bookmarks]{hyperref}
\usepackage[dvipsnames]{xcolor}
\hypersetup{colorlinks=true,citecolor=NavyBlue,linkcolor=Brown,urlcolor=Orange}

\usepackage{tikz-cd,adjustbox}
\usepackage[arrow,curve,matrix]{xy}

\newif\ifdraft
\draftfalse

\counterwithin{equation}{subsection}
\counterwithout{subsection}{section}
\counterwithin{figure}{subsection}

\newtheorem{theorem}[equation]{Theorem}
\newtheorem*{theorem*}{Theorem}
\newtheorem{lemma}[equation]{Lemma}
\newtheorem*{lemma*}{Lemma}
\newtheorem{corollary}[equation]{Corollary}
\newtheorem{proposition}[equation]{Proposition}
\newtheorem*{proposition*}{Proposition}

\theoremstyle{definition}
\newtheorem{definition}[equation]{Definition}
\newtheorem*{definition*}{Definition}
\newtheorem{remark}[equation]{Remark}

\newtheorem{example}[equation]{Example}
\newtheorem*{example*}{Example}
\newtheorem*{problem*}{Problem}

\theoremstyle{plain}

\newcounter{intro}

\newtheorem{intro-conjecture}[intro]{Conjecture}
\newtheorem{intro-corollary}[intro]{Corollary}
\newtheorem{intro-theorem}[intro]{Theorem}
\newtheorem{intro-proposition}[intro]{Proposition}

\newcommand{\theoremref}[1]{\hyperref[#1]{Theorem~\ref*{#1}}}
\newcommand{\lemmaref}[1]{\hyperref[#1]{Lemma~\ref*{#1}}}
\newcommand{\definitionref}[1]{\hyperref[#1]{Definition~\ref*{#1}}}
\newcommand{\propositionref}[1]{\hyperref[#1]{Proposition~\ref*{#1}}}
\newcommand{\conjectureref}[1]{\hyperref[#1]{Conjecture~\ref*{#1}}}
\newcommand{\corollaryref}[1]{\hyperref[#1]{Corollary~\ref*{#1}}}
\newcommand{\exampleref}[1]{\hyperref[#1]{Example~\ref*{#1}}}

\DeclareMathOperator{\Pic}{Pic}

\DeclareMathOperator{\gr}{gr}
\DeclareMathOperator{\MHM}{MHM}
\DeclareMathOperator{\dr}{DR}
\DeclareMathOperator{\ic}{IC}

\DeclareMathOperator{\lcdef}{lcdef}
\DeclareMathOperator{\cdef}{cdef}
\DeclareMathOperator{\pdef}{pdef}
\DeclareMathOperator{\Supp}{Supp}
\DeclareMathOperator{\depth}{depth}

\DeclareMathOperator{\Cone}{Cone}
\DeclareMathOperator{\coh}{coh}

\DeclareMathOperator{\red}{red}
\DeclareMathOperator{\anloc}{an-loc}
\DeclareMathOperator{\rat}{rat}

\DeclareMathOperator{\codim}{codim}
\DeclareMathOperator{\RHD}{RHD}

\makeatletter
\let\old@caption\caption
\renewcommand*{\caption}[1]{%
	\setcounter{figure}{\value{equation}}%
	\stepcounter{equation}%
	\old@caption{#1}\relax%
}
\makeatother

\begin{document}

\title{Local Cohomological Defect and a Conjecture of Musta{\c{t}}{\u{a}}-Popa}

\author{Andrew Burke}
\address{Department of Mathematics, Harvard University, 1 Oxford Street, Cambridge, MA 02138, USA}
\email{aburke@math.harvard.edu}

\begin{abstract}
We prove a general result on the depth of Du Bois complexes of a singular variety. We apply it to prove a conjecture of Musta{\c{t}}{\u{a}}-Popa and to study the local cohomological defect, extending results of Ogus and Dao-Takagi over the complex numbers.
\end{abstract}

\maketitle

\setcounter{tocdepth}{1}

\section{Introduction}

The present paper is dedicated to new results around the local cohomological defect, a singularity invariant which has received growing interest in recent years. It admits topological \cite{bbl}, \cite{ogus}, \cite{saito lcdef} and holomorphic characterizations \cite{local cohomology}. Additionally, it controls Lefschetz-type theorems on singular varieties \cite{popa park lefschetz}. 

Throughout, let $X$ be a complex algebraic variety of dimension $n$. If $X$ is embedded in a smooth variety $Y$ of dimension $d$, consider the local cohomology sheaves $\mathcal{H}^q_X(\mathscr{O}_Y)$. The first nonzero sheaf occurs at $q=c$, where $c = d-n$. We define the \emph{local cohomological defect}
\begin{align*}
    \lcdef(X) = \max \{ q : \mathcal{H}^q_X( \mathscr{O}_Y) \neq 0 \} - c.
\end{align*}
It satisfies $0 \leq \lcdef(X) \leq n$. If $X$ has quotient or local complete intersection singularities, then $\lcdef(X) = 0$; see \cite[Example $2.4$]{popa shen} for more examples.

One can express $\lcdef(X)$ in terms of more refined holomorphic invariants. Denote by $\underline{\Omega}_X^p \in D^b_{\coh}(X)$ the \emph{$p$-th Du Bois complex} \cite{du bois} of $X$. This is a Hodge-theoretic replacement for the sheaf of holomorphic $p$-forms on a smooth variety, which can be computed in terms of a hyperresolution. Then \cite[Theorem E]{local cohomology} proves

\begin{align}\label{intro lcdef depth}
    \lcdef(X) = n - \min_{0 \leq p \leq n} \{ \depth \underline{\Omega}_X^p + p \}.
\end{align}
In particular, $\lcdef(X)$ is independent of embedding. A priori, one must compute $\depth \underline{\Omega}_X^p$ for all $p \leq n$ to specify $\lcdef(X)$. An application of the main result of this paper is that it suffices to check only $p \leq \lceil (n-3)/2 \rceil$. In fact, more is true:

\begin{intro-theorem}\label{intro thm easy}
Let $k$ be an integer. Then we have $\lcdef (X) \leq n-k$ if and only if 
$\depth \underline{\Omega}_X^p + p \geq k$ for all $p \leq \lceil (k-3)/2 \rceil$.
\end{intro-theorem}

When $k \leq 3$, we recover (slight extensions of) well-known results of Ogus \cite{ogus} and Dao-Takagi \cite{dao takagi}: if $\depth \mathscr{O}_X \geq k$, then $\lcdef(X) \leq n-k$. Note that $\depth \underline{\Omega}_X^0 \geq \depth \mathscr{O}_X$ by \cite{kovacs schwede}. Example \ref{exmp beat dao takagi} shows statements involving $\depth \underline{\Omega}_X^0$ are strictly stronger than those involving $\depth \mathscr{O}_X$. The obvious pattern fails beyond $k = 3$; see \cite[Example $2.11$]{dao takagi}. Our result explains this failure and completes the picture.

Theorem \ref{intro thm easy} is a consequence of a more general result about the depth of Du Bois complexes, which is our main theorem.

\begin{intro-theorem}\label{intro thm main}
Let $k$ and $m$ be integers. If
\begin{align*}
    \depth \underline{\Omega}_X^p + p \geq k \hspace{.2cm} \text{ for all } p \leq m,
\end{align*}
then the same inequality holds for $p \geq k-m-2$.
\end{intro-theorem}

Another application of Theorem \ref{intro thm main} is the proof of a conjecture of Musta{\c{t}}{\u{a}}-Popa, which arose from trying to place the results of Ogus and Dao-Takagi in the context of the Hodge filtration on local cohomology in \cite{local cohomology}. Suppose $X$ is embedded in a smooth variety $Y$ of dimension $d$. Choose a log resolution $f:Z \rightarrow Y$ of the pair $(Y, X)$, assumed to be an isomorphism away from $X$. Denote $E = f^{-1}(X)_{\red}$, which is a simple normal crossings divisor on $Z$.

\begin{intro-conjecture}\cite[Conjecture $14.1$]{local cohomology}\label{intro conj}
If $\depth \mathscr{O}_X  \geq j+2$, then $R^{d-2} f_* \Omega_Y^{d-j} ( \log E) = 0$.
\end{intro-conjecture}

Although it is phrased without reference to Hodge theory, the conclusion is equivalent to $\depth \underline{\Omega}_X^j \geq 2$ by \cite[Lemma $14.4$]{local cohomology}. Plugging $m = 0$ into Theorem \ref{intro thm main} gives the following slightly stronger statement.

\begin{intro-corollary}\label{intro cor answering conj}
If $\depth \underline{\Omega}_X^0 \geq j+2$, then $\depth \underline{\Omega}_X^j \geq 2$.
\end{intro-corollary}
The conjecture was previously known when $X$ has isolated singularities \cite[Theorem $14.3$]{local cohomology} or rational singularities outside of a finite set of points \cite[Theorem $0.7$]{weaklyratl}. 

Theorem \ref{intro thm easy} also implies a bound for the local cohomological defect in Corollary \ref{cor lcdef vs q} in terms of the smallest index $p(X)$ for which the minimum in \ref{intro lcdef depth} is obtained. This strictly improves the bound of \cite[Corollary $7.5$]{popa park lefschetz} for $\lcdef(X)$ (see Example \ref{exmp diamond}), and it sometimes determines which $\depth \underline{\Omega}_X^p$ computes $\lcdef(X)$.

\begin{intro-theorem}\label{intro dt thm}
Let $m$ be an integer and suppose $\depth \underline{\Omega}_X^p +p \geq n$ for all $p \leq m-1$. If $\depth \underline{\Omega}_X^m \leq m+3$, then
\begin{align*}
    \lcdef(X) = \max \{ 0, n-\depth \underline{\Omega}_X^m -m \}.
\end{align*}
The first assumption holds if $X$ has pre-$(m-1)$-rational singularities.
\end{intro-theorem}

In particular, $\lcdef(X) = \max \{ 0, n-\depth \underline{\Omega}_X^0 \}$ when $\depth \underline{\Omega}_X^0 \leq 3$. This equality fails beyond depth three. Nevertheless, Dao-Takagi \cite{dao takagi} prove when $\depth \mathscr{O}_X \geq 4$ that $\lcdef(X) \leq n-4$ if and only if $X$ has torsion local analytic Picard groups. More generally, our final result is a technical criterion for $\lcdef(X) \leq n-k$, assuming $\depth \underline{\Omega}_X^p + p \geq k$ for $p \leq m$, where $m$ is fixed. It is phrased in terms of a family of invariants parameterized by integers $0 \leq d \leq n$, which we call the \emph{perversity defects} $\pdef(X, d)$ of $X$ (see Definition \ref{defn pdef}). For example, $\pdef(X,n) = \lcdef(X)$, while $\pdef(X,0)$, called the \emph{constructible defect}, is a measure of the constructible cohomologies of $\mathbb{D} \mathbb{Q}_X$. These invariants can be expressed in terms of the cohomology of small punctured neighborhoods of points, as in Lemma \ref{lemma link}.

\begin{intro-theorem}\label{intro thm pdef}
Let $k$ and $m$ be integers. We have $\lcdef(X) \leq n-k$ if and only if $\depth \underline{\Omega}_X^p + p \geq k$ for all $p \leq m$ and $\pdef(X, k-2m-4) \leq n-k$.
\end{intro-theorem}

Theorem \ref{intro thm pdef} implies the aforementioned result of Dao-Takagi (see Example \ref{example dt}) and an analogous criteria for $\lcdef(X) \leq n-5$ by Park-Popa \cite[Corollary G]{popa park lefschetz} (see Example \ref{exmp lcdef 5}). 

Our approach relies throughout on the theory of mixed Hodge modules. The main technical result is Proposition \ref{prop main}, an abstract vanishing statement for graded de Rham pieces of mixed Hodge modules. It yields our main Theorem \ref{intro thm main}, because Du Bois complexes can be identified with shifts of graded de Rham pieces of the trivial Hodge module. We also provide a direct explanation (Theorem \ref{thm lcdef depth body}) for the equivalence between the topological and holomorphic characterizations of $\lcdef(X)$, as was asked about in \cite[Remark $11.2$]{local cohomology}.

\subsection*{Acknowledgements}
I would like to express my sincere thanks to my advisor, Mihnea Popa, for his unwavering support and guidance. I also want to thank Hyunsuk Kim, Sung Gi Park, Wanchun Shen, and Anh Duc Vo for helpful discussions.

\section{Preliminaries}

\subsection{Mixed Hodge modules}
This paper will extensively use the language of mixed Hodge modules \cite{saito 88, saito 90}. We think of a mixed Hodge module as a perverse sheaf equipped with additional structure, including a Hodge filtration and a weight filtration. The category $\MHM(X)$ of mixed Hodge modules is abelian, so it admits a derived category $D^b(\MHM(X))$.

There is a faithful realization functor 
\begin{align*}
    \rat_X: D^b(\MHM(X)) \rightarrow D^b_{c}(X)
\end{align*}
sending a mixed Hodge module to its underlying perverse sheaf, considered as a complex of constructible sheaves. The usual functors $f_*, f^*, f_!, f^!, \otimes$, and $\mathbb{D}$ on $D^b_c(X)$ lift to $D^b(\MHM(X))$.

Let $M$ be a mixed Hodge module on $X$, and suppose $X$ is locally embedded in a smooth variety $Y$ of dimension $d$. Denote by $\mathcal{M}$ the $\mathcal{D}_Y$-module corresponding to the realization of $M$ under the Riemann-Hilbert correspondence. The Hodge filtration is a filtration $F_\bullet \mathcal{M}$ on $\mathcal{M}$. Let us emphasize that the filtered object $F_\bullet \mathcal{M}$ is defined on the ambient space $Y$, not on $X$. The more invariant global objects to consider are the graded de Rham pieces of $M$, defined locally by the formula
\begin{align}\label{eqn grDR}
    \gr_{p}^F \dr(M) = \big( \gr_{p-d}^F M \otimes \wedge^m T_Y \rightarrow \ldots \rightarrow \gr_{p-1}^F M \otimes T_Y \rightarrow \gr_{p}^F M \big).
\end{align}
They are well-defined complexes of coherent sheaves on $X$ by \cite[Proposition $2.33$]{saito 90}. Moreover, $\gr_p^F \dr$ is an exact functor, so it extends to 
\begin{align*}
    \gr_p^F \dr : D^b(\MHM(X)) \rightarrow D^b_{\coh}(X).
\end{align*}
One computes $\gr_p^F \dr( M^\bullet) $ for a complex $M^\bullet$ by applying \ref{eqn grDR} to each term.

Mixed Hodge modules are also equipped with a weight filtration $W_\bullet$ on the underlying perverse sheaf. One says a complex $M^\bullet \in D^b(\MHM(X))$ has weights $\geq c$ (resp. $\leq c$) if $\gr_\ell^W \mathcal{H}^j M^\bullet = 0$ for all $\ell - j < c$ (resp. $\ell - j \geq c$).

\subsection{Trivial Hodge module}

Let $X$ be a complex algebraic variety of dimension $n$. We define the \emph{trivial Hodge module} $\mathbb{Q}_X^H[n] \in D^b(\MHM(X))$ to be the complex $a^* \mathbb{Q}^H[n]$, where $a: X \rightarrow \mathop{\mathrm{Spec}} \mathbb{C}$. It has cohomologies in degrees $\leq 0$ and has weights $\leq n$. 

The \emph{intersection cohomology Hodge module} $\ic_X^H$ is $\gr_n^W \mathcal{H}^0 \mathbb{Q}_X^H[n]$, a pure Hodge module  of weight $n$ (see \cite[$4.5$]{saito 90}). Denote by $\gamma_X$ the composition of the natural morphisms $\mathbb{Q}_X^H[n] \rightarrow \mathcal{H}^0 \mathbb{Q}_X^H[n] \rightarrow \gr_n^W \mathcal{H}^0 \mathbb{Q}_X^H[n] = \ic_X^H$. The \emph{RHM-defect object} $\mathcal{K}_X^\bullet \in D^b(\MHM(X))$ is the complex sitting in the distinguished triangle
\begin{align}\label{eqn rhm triangle}
    \mathcal{K}_X^\bullet \rightarrow \mathbb{Q}_X^H[n] \xrightarrow[]{\gamma_X} \ic_X^H \xrightarrow[]{+1}
\end{align}
See \cite[Definition $6.1$]{popa park lefschetz} and \cite[Proposition $6.4$]{popa park lefschetz}, where the following is proved.

\begin{lemma}\label{lemma wt n-1}
$\mathcal{K}_X^\bullet$ has weights $\leq (n-1)$
\end{lemma}

\subsection{Du Bois Complexes}

Saito \cite[Theorem $0.2$]{saito hodge complexes} proved that the Du Bois complexes of $X$ can be identified with the graded de Rham pieces of the trivial Hodge module, up to shift:
\begin{align*}
    \underline{\Omega}_X^p \cong \gr_{-p}^F \dr \big( \mathbb{Q}_X^H[n] \big) [p-n].
\end{align*}
We make this identification throughout the paper. Analogously, the \emph{$p$-th intersection Du Bois complex} \cite{psv} of $X$ is the complex $I \underline{\Omega}_X^p = \gr_{-p}^F \dr( \ic_X^H )[p-n]$. These complexes are nonzero only for $0 \leq p \leq n$. Moreover, there are natural maps $\Omega_X^p \rightarrow \underline{\Omega}_X^p \rightarrow I \underline{\Omega}_X^p$, which are isomorphisms when $X$ is smooth.

For future reference, we record that since $\ic_X^H$ is a pure Hodge module of weight $n$,
\begin{lemma}\label{lem ic vanishing}
    $\mathcal{H}^j \gr_p^F \dr ( \ic_X^H ) = 0$, whenever $j > 0$, $p > 0$, or $p < -n$.
\end{lemma}

Recently, there has been a surge of interest in singularity notions related to the Du Bois complex. Especially important are generalizations of rational and Du Bois singularities called $m$-rational and $m$-Du Bois singularities. We will not need much from this subject, except for a couple of definitions. Say $X$ has \emph{pre-$m$-rational singularities} \cite{shen venkatesh vo} if 
\begin{align*}
    \mathcal{H}^i ( \mathbb{D} ( \underline{\Omega}_X^{n-p})[-n] ) = 0 \hspace{.5cm} \text{ for all } i > 0 \text{ and } 0 \leq p \leq m.
\end{align*}
We denote $\mathbb{D}(-) = \mathcal{R}\mathcal{H}om_{\mathscr{O}_X}(-, \omega_X^\bullet)$, where $\omega_X^\bullet$ is the dualizing complex. When $X$ is smooth, this is isomorphic to $\omega_X[n]$. Say that $X$ satisfies \emph{condition $D_m$} if $\underline{\Omega}_X^p \rightarrow I \underline{\Omega}_X^p$ is an isomorphism for all $p \leq m$. See \cite{popa park lefschetz} or \cite{dod}, where an equivalent condition is studied. If $X$ has pre-$m$-rational singularities, then it satisfies condition $D_m$.

\subsection{Depth and local cohomological defect}

We require a notion of depth for complexes of coherent sheaves $E \in D^b_{\coh}(X)$. Denote
\begin{align*}
    \depth E = -\max \{ j : \mathcal{H}^{j} \mathbb{D}E \neq 0 \},
\end{align*}
 If $E$ is a coherent sheaf, this agrees with the usual definition of $\depth E$. 

We will study $\depth \underline{\Omega}_X^p$ in detail. Note that $\mathbb{D} \underline{\Omega}_X^p = \gr_p^F \dr (\mathbb{D} \mathbb{Q}_X^H)[-p]$, because $\mathbb{D} \circ \gr_{-p}^F \dr = \gr_p^F \dr \circ \mathbb{D}$ by \cite[Section $2.4$]{saito 88} or \cite[Lemma $3.2$]{park}. Hence
\begin{align}\label{eqn depth du bois}
    \depth \underline{\Omega}_X^p + p = -\max \{ j :  \mathcal{H}^j \gr_p^F \dr (\mathbb{D} \mathbb{Q}_X^H) \neq 0 \}.
\end{align}
We have the basic inequalities $0 \leq \depth \underline{\Omega}_X^p \leq n$. In fact, $\depth \underline{\Omega}_X^p \geq 1$, as $\mathcal{H}^0\underline{\Omega}_X^p$ is torsion-free by \cite[Remark $3.8$]{h-topology}. Moreover, $\depth I\underline{\Omega}_X^p \geq n-p$, because $\ic_X^H$ is self-dual, up to twist, and the graded de Rham pieces of a Hodge module are concentrated in non-positive degrees. From this, we deduce

\begin{lemma}\label{lem condtion Dm}
If $X$ satisfies condition $D_m$, then $\depth \underline{\Omega}_X^p +p \geq n$ for all $p \leq m$. 
\end{lemma}

We also require the injectivity theorem of \cite{kovacs schwede}. It says the maps
\begin{align*}
    \mathscr{E}xt^j_{\mathscr{O}_X}(\underline{\Omega}_X^0, \omega_X^\bullet) \rightarrow \mathscr{E}xt^j_{\mathscr{O}_X}(\mathscr{O}_X, \omega_X^\bullet).
\end{align*}
induced from $\mathscr{O}_X \rightarrow \underline{\Omega}_X^0$ are injective for each $j$. In particular, 
\begin{align}\label{eqn depth inequ}
   \depth \underline{\Omega}_X^0 \geq \depth \mathscr{O}_X. 
\end{align}

Recall that the local cohomological defect $\lcdef(X)$ is defined using local cohomology sheaves. If $X$ is embedded in a smooth variety $Y$, then
\begin{align*}
    \lcdef(X) = \max \{ q : \mathcal{H}^q_X( \mathscr{O}_Y) \neq 0 \} - c,
\end{align*}
where $c = \dim Y - \dim X$. It admits a topological interpretation in terms of the \emph{rectified $\mathbb{Q}$-homological depth} $\RHD_\mathbb{Q}(X) = \min \{ j : {}^p \mathcal{H}^j(\mathbb{Q}_X) \neq 0 \}$. Indeed, \cite{saito lcdef} proves
\begin{align}\label{eqn lcdef top}
    \lcdef(X) = n - \RHD_\mathbb{Q}(X).
\end{align}
Since duality is exact on perverse sheaves, we have $\mathbb{D} ({}^p\mathcal{H}^{j}(\mathbb{Q}_X)) = {}^p\mathcal{H}^{-j}(\mathbb{D}\mathbb{Q}_X)$. If we combine this with the fact that the realization functor for mixed Hodge modules is faithful, we deduce
\begin{align}\label{eqn lcdef}
    \lcdef(X)= \max \{ j :  \mathcal{H}^j(\mathbb{D} (\mathbb{Q}_X^H[n])) \neq 0 \}.
\end{align}
We will use this formula for $\lcdef(X)$ throughout the paper.

\subsection{Perversity functions and perversity defect}
The category of perverse sheaves is the heart of the so-called middle perversity t-structure on the derived category $D^b_c(X)$ of constructible sheaves on $X$. In this section, we recall more generally how a perversity function determines a t-structure on $D^b_c(X)$. See \cite{bbd} or \cite{kashiwara} for details. We also introduce variants of the local cohomological defect with respect to such perversities.

Fix a function $p: \{ 0,1, \ldots , n \} \rightarrow \mathbb{Z}$, such that $p(k) - p(k+1)$ is $0,1$, or $2$ for all $k$. This is called a \emph{perversity function}.

\begin{definition}
Define a pair $({}^p D_c^{\leq 0}(X), {}^p D_c^{\geq 0}(X))$ of full subcategories of $D_c^b(X)$ as follows. For $K \in D_c^b(X)$, say $K \in {}^p D^{\leq 0}(X)$ if
\begin{align*}
    \dim \{ \Supp \mathcal{H}^j (K) \} < k, \hspace{.5cm} \text{if } j > p(k),
\end{align*}
and say $K \in {}^p D^{\geq 0}(X)$ if
\begin{align*}
    \dim \{ \Supp \mathcal{H}^{-j} ( \mathbb{D}  K) \} < k, \hspace{.5cm} \text{if } j < p(k) + 2k.
\end{align*}
\end{definition}
These subcategories form a t-structure on $D_c^b(X)$ by \cite[Theorem $10.2.8$]{kashiwara}. There is a convenient re-formulation of the definition in terms of a choice of Whitney stratification.

\begin{proposition}\cite[Proposition $10.2.4$]{kashiwara}\label{prop locally constant}
Let $K \in D_c^b(X)$ and consider a Whitney stratification of $X$ such that $i_{S}^* K$ and $i_{S}^! K$ have locally constant cohomology sheaves for each inclusion of strata $i_S: S\rightarrow X$. Then:
\begin{enumerate}
    \item $K \in {}^p D^{\leq 0}_c(X)$ if and only if $\mathcal{H}^j(i_{S}^* K) = 0$, whenever $j > p(d_S)$,
    \item $K \in {}^p D_c^{\geq 0}(X)$ if and only if $\mathcal{H}^j(i_{S}^! K) = 0$, whenever $j < p(d_S)$.
\end{enumerate}
\end{proposition}

The more familiar t-structures on $D^b_c(X)$ are special cases of this construction.

\begin{example}\label{example cdef}
The function $p_0(x) = 0$ is called the \emph{zero perversity}. It determines the usual constructible t-structure on $D^b_c(X)$.
\end{example}

\begin{example}\label{example lcdef}
The function $p_n(x) = -x$ is called the \emph{middle perversity}. It determines the middle perversity t-structure, whose heart consists of perverse sheaves. Unless otherwise specified, the letter $p$ will henceforth refer specifically to the middle perversity.
\end{example}

We require a family of perversities interpolating between the zero perversity and the middle perversity.

\begin{example}
Consider, for each integer $0 \leq d \leq n$, the perversity function
\begin{align*}
    p_d(x) = - \min \{ x, d\}
\end{align*}
By an abuse of notation, we denote the associated t-structure by $({}^{d} D_c^{\leq 0}(X), {}^{d} D_c^{\geq 0}(X))$ and its cohomology sheaves by ${}^d \mathcal{H}^j(-)$. Observe that the case $d=0$ is the zero perversity and $d = n$ is the middle perversity.
\end{example}

Recall from \ref{eqn lcdef} that $\lcdef(X)$ is the largest nonzero cohomology of $\mathbb{D}(\mathbb{Q}_X[n])$ with respect to the middle perversity. We define a generalization of $\lcdef(X)$, in terms of the perversities $p_d$ for $0 \leq d \leq n$.

\begin{definition}\label{defn pdef}
The \emph{$d$-th perversity defect} of $X$ is
\begin{align*}
    \pdef(X, d) = \max \{ j : {}^d\mathcal{H}^j( \mathbb{D}(\mathbb{Q}_X[n])) \neq 0 \}.
\end{align*}
In other words, $\pdef(X,d)$ is the minimum index $j$ such that $\mathbb{D}(\mathbb{Q}_X[n]) \in {}^d D_c^{\leq j}(X)$. 
\end{definition}

By Example \ref{example lcdef}, $\pdef(X,n) = \lcdef(X)$ . We call $\pdef(X, 0)$ the \emph{constructible defect} $\cdef(X)$ of $X$. It is equal to the largest nonzero constructible cohomology of $\mathbb{D}(\mathbb{Q}_X[n])$ by Example \ref{example cdef}.

There is a chain of inequalities of perversity functions $p_n \leq \ldots \leq p_1 \leq p_0$. It yields inclusions ${}^n D^{\leq j}_c(X) \subset \ldots \subset {}^1 D^{\leq j}_c(X) \subset {}^0 D^{\leq j}_c(X)$, and hence
\begin{align}\label{eqn inequality chain}
    \cdef(X) = \pdef(X, 0) \leq \pdef(X, 1) \leq \ldots \leq \pdef(X,n) = \lcdef(X).
\end{align}

We can also also interpret the perversity defects $\pdef(X,d)$ in terms of the cohomology of small punctured neighborhoods. Indeed, fix a Whitney stratification of $X$. By the local structure of a Whitney stratification, $\mathbb{D}(\mathbb{Q}_X[n])$ has locally constant cohomology sheaves along each strata. Proposition \ref{prop locally constant} implies $\pdef(X,d) \leq n-k$ if and only $\mathcal{H}^j(i_S^* \mathbb{D}( \mathbb{Q}_X[n]) ) = 0$ for each strata $S$ and $j > n-k  - \min \{ d_S, d\}$. For a point $x \in S$, if $i_x: \{ x\} \rightarrow X$, we have

\begin{align*}
    \mathcal{H}^j (i_x^* \mathbb{D}(\mathbb{Q}_X[n])) = H^{-j}(i_x^! \mathbb{Q}_X[n]) = H^{n-j}_{\{x\}}(X, \mathbb{Q}).
\end{align*}

Let us denote by $U$ a Stein analytic neighborhood of $x$. There is a long exact sequence
\begin{align*}
    \ldots \rightarrow H^{n-j-1}(U \setminus \{ x \}, \mathbb{Q}) \rightarrow H^{n-j}_{\{x\}}(X, \mathbb{Q}) \rightarrow H^{n-j}(U, \mathbb{Q}) \rightarrow \ldots 
\end{align*}
As $U$ is contractible, this gives isomorphisms $H^{n-j}_{\{x\}}(X, \mathbb{Q}) \simeq \tilde{H}^{n-j-1}(U \setminus \{ x\}, \mathbb{Q})$. We conclude:

\begin{lemma}\label{lemma link}
If $k$ and $d$ are integers, we have $\pdef(X, d) \leq n - k$ if and only if
\begin{align*}
    \Tilde{H}^i(U \setminus \{ x\}, \mathbb{Q}) = 0 \hspace{.5cm} \text{ for all } i \leq k-2 + \min \{ d_S, d\} \text{ and all } x \in X.
\end{align*}   
\end{lemma}
For example, $\lcdef(X) \leq n-k$ if and only if $\Tilde{H}^i(U \setminus \{ x\}, \mathbb{Q}) = 0$ for all $i \leq k+d_S-2$ and all $x \in X$, as was shown in \cite[Corollary $3$]{saito lcdef} in terms of links. Moreover, $\cdef(X) \leq n-k$ if and only if $\Tilde{H}^i(U \setminus \{ x\}, \mathbb{Q}) = 0$ for all $i \leq k-2$ and all $x \in X$. This characterization gives another perspective on the inequalities \ref{eqn inequality chain}.

Let us offer one final way of thinking about the constructible defect $\cdef(X) = \pdef(X,0)$. One can view the trivial Hodge module $\mathbb{Q}_X^H$ as an iterated extension of Du Bois complexes
\begin{align*}
    \mathbb{Q}_X^H = \bigg( \underline{\Omega}_X^0 \rightarrow \underline{\Omega}_X^1 \rightarrow \ldots \rightarrow \underline{\Omega}_X^n \bigg)
\end{align*}
inside the derived category of differential complexes on $X$. Dualizing this expression, we obtain
\begin{align*}
    \mathbb{D}( \mathbb{Q}_X^H[n] ) = \bigg( \mathbb{D} \underline{\Omega}_X^n \rightarrow \ldots \rightarrow \mathbb{D} \underline{\Omega}_X^1 \rightarrow \mathbb{D} \underline{\Omega}_X^0 \bigg).
\end{align*}
It gives rise to a cohomology spectral sequence
\begin{align*}
    E_1^{p,q} = \mathcal{H}^q \mathbb{D} \underline{\Omega}_X^{n-p} \Rightarrow {}^c \mathcal{H}^{p+q} \mathbb{D}(\mathbb{Q}_X[n]).
\end{align*}
Using \ref{intro lcdef depth}, we have $\lcdef(X) \leq n-k$ if and only if $E_1^{p,q} = 0$ for all $p+q > n-k$. On the other hand, $\cdef(X) \leq n-k$ if and only if ${}^c \mathcal{H}^{p+q} \mathbb{D}(\mathbb{Q}_X^H[n]) = 0$ for all $p+q>n-k$. 

If the spectral sequence degerates at the $E_1$ page, then clearly $\cdef(X) = \lcdef(X)$. Very roughly, we can view the difference between $\lcdef(X)$ and $\cdef(X)$ as a measure of the failure of degeneration.

\begin{remark}
Though we will not use it in what follows, we remark that for each perversity function $p$, there is an associated t-structure $({}^p D^{\leq 0}(\MHM(X)), {}^p D^{\geq 0}(\MHM(X))$ on the derived category of mixed Hodge modules which lifts  $({}^p D_c^{\leq 0}(X), {}^p D_c^{\geq 0}(X))$. More precisely, a complex $\mathcal{M}^\bullet$ lies in ${}^p D^{\leq 0}(\MHM(X))$ (resp. ${}^p D^{\geq 0}(\MHM(X))$) if and only if $\rat_X(M^\bullet)$ lies in ${}^p D_c^{\leq 0}(X)$ (resp. ${}^p D_c^{\geq 0}(X)$). This follows from the argument of \cite[Theorem $1.2.8$]{kashiwara}, which only uses the six functor formalism. Saito remarked that the constructible t-structure lifts to mixed Hodge modules in \cite[Remark $4.6.2$]{saito 90}.
\end{remark}

\section{Depth and Local Cohomological Defect}\label{depth section}

\subsection{Fundamental Lemmas}
In this subsection, we prove some technical results which will be the key to our analysis.

Fix an integer $m$ and an object $M^\bullet \in D^b(\MHM(X))$. Locally, $X$ is embeddable, and \ref{eqn grDR} shows the condition $\gr_p^F \dr (M^\bullet) = 0$ for all $p \leq m$ is equivalent to $F_m M^\bullet = 0$. We need a refinement of this basic observation that accounts for cohomologies.

\begin{lemma}\label{lemma hodge-cohom 1}
Suppose $X$ is embeddable. Fix integers $m$ and $c$ and an object $M^\bullet \in D^b(\MHM(X))$. Then

\begin{align*}
    \mathcal{H}^{j} \gr_p^F \dr( M^\bullet) = 0 \hspace{.5cm} \mathrm{for} \hspace{.1cm} \mathrm{all} \hspace{.2cm} j \geq c \hspace{.2cm} \mathrm{and} \hspace{.2cm} p \leq m
\end{align*}
if and only if
\begin{align*}
    \mathcal{H}^{j} \gr_p^F (M^\bullet) = 0 \hspace{.5cm} \mathrm{for} \hspace{.1cm} \mathrm{all} \hspace{.2cm} j \geq c \hspace{.2cm} \mathrm{and} \hspace{.2cm} p \leq m.
\end{align*}
In particular, if $m \gg 0$, this latter condition is equivalent to $\mathcal{H}^j M^\bullet = 0$ for all $j \geq c$.
\end{lemma}

\begin{proof}
Assume $\mathcal{H}^j \gr_p^F (M^\bullet) = 0$ for all $j \geq c$ and $p \leq m - 1$,. It suffices to prove
\begin{align*}
    \mathcal{H}^{j} \gr_m^F \dr( M^\bullet) \cong \mathcal{H}^{j} \gr_m^F ( M^\bullet) \hspace{.5cm} \text{ for all } j \geq c.
\end{align*}

Consider the expression
\begin{align*}
    \gr_{m}^F \dr(M^\bullet) = \big( \gr_{m-d}^F M^\bullet \otimes \wedge^d T_Y \rightarrow \ldots \rightarrow \gr_{m-1}^F M^\bullet \otimes T_Y \rightarrow \gr_{m}^F M^\bullet \big).
\end{align*}
There is a convergent spectral sequence with $E_1$ page
\begin{align*}
    E_1^{p,q} = \mathcal{H}^{q}(\gr_{m+p}^F M^\bullet \otimes \wedge^{-p} T_Y  ) \implies \mathcal{H}^{p+q} \gr_m^F \dr(M^\bullet) = 0.
\end{align*}
The terms $E_1^{p,q}$ vanish if $p> 0$ or if $p < 0$ and $q \geq c$ by assumption. One deduces that $E_1^{0,j} = E_\infty^{0,j}$ for $j \geq c$ and, in fact,
\begin{align*}
    \mathcal{H}^{j} \gr_m^F \dr( M^\bullet) \cong \mathcal{H}^{j} \gr_m^F ( M^\bullet)
\end{align*}
as desired.

The last statement follows, because $\gr_p^F$ commutes with taking cohomology.
\end{proof}

\begin{remark}
Though we will not use it in this paper, a similar spectral sequence argument proves: $\mathcal{H}^{j} \gr_p^F \dr( M^\bullet) = 0$ for all $j \leq c-p$ and $p \leq m$ if and only if $\mathcal{H}^{j} \gr_p^F( M^\bullet) = 0$ for all $j \leq c-p$ and $p \leq m$.
\end{remark}

We will often reference the following immediate corollary of Lemma \ref{lemma hodge-cohom 1}.

\begin{corollary}\label{cor hodge-cohom 1}
\begin{align*}
    \mathcal{H}^{j} \gr_p^F \dr( M^\bullet) = 0 \hspace{.5cm} \mathrm{for} \hspace{.1cm} \mathrm{all} \hspace{.2cm} j \geq c \hspace{.2cm} \mathrm{and} \hspace{.2cm} p \leq m
\end{align*}
if and only if
\begin{align*}
     \gr_p^F \dr (\mathcal{H}^j M^\bullet) = 0 \hspace{.5cm} \mathrm{for} \hspace{.1cm} \mathrm{all} \hspace{.2cm} j \geq c \hspace{.2cm} \mathrm{and} \hspace{.2cm} p \leq m.
\end{align*}
\end{corollary}

\begin{proof}
Apply Lemma \ref{lemma hodge-cohom 1} to $M^\bullet$ and each $\mathcal{H}^j M^\bullet$ for $j \geq c$ separately, and use the fact that $\gr_p^F$ commutes with taking cohomology.
\end{proof}

We also require a lemma which brings weights into the picture.

\begin{lemma}\label{lemma hodge-weight}
Let $M$ be a mixed Hodge module on $X$ of weight $\geq \ell$. If
\begin{align*}
    \gr_{p}^F \dr(M) = 0 \hspace{.5cm} \text{ for all } p \leq m
\end{align*}
then the same holds for $p \geq -\ell - m$.
\end{lemma}

\begin{proof}
This is precisely the dual of \cite[Proposition $5.2$]{popa park lefschetz}. 
\end{proof}

We can now prove our main abstract result.

\begin{proposition}\label{prop main}
Fix integers $c, m$, and $w$. Let $M^\bullet \in D^b(\MHM(X))$ and suppose $M^\bullet$ has weights $\geq w$. If
\begin{align*}
    \mathcal{H}^j \gr_i^F \dr(M^\bullet) = 0 \hspace{.2cm} \text{ for all } j \geq c \text{ and } i \leq m,
\end{align*}
then
\begin{align*}
    \mathcal{H}^j \gr_i^F \dr(M^\bullet) = 0 \hspace{.2cm} \text{ for all } j \geq c \text{ and } i \geq -c-m-w.
\end{align*}
\end{proposition}

\begin{proof}
By Corollary \ref{cor hodge-cohom 1}, our assumption is equivalent to the vanishing
\begin{align*}
    \gr_i^F \dr( \mathcal{H}^j  M^\bullet) = 0 \hspace{.2cm} \text{ for all } j \geq c \text{ and } i \leq m.
\end{align*}
Since the mixed Hodge module $\mathcal{H}^j M^\bullet$ has weights $\geq (w+j)$, Lemma \ref{lemma hodge-weight} shows
\begin{align*}
    \gr_i^F \dr( \mathcal{H}^j M^\bullet) = 0 \hspace{.2cm} \text{ for all } j \geq c \text{ and } i \geq -j-m-w.
\end{align*}
Fix $i \geq -c-m-w$, and consider the convergent spectral sequence
\begin{align*}
    E_2^{p,q} = \mathcal{H}^p \gr_i^F \dr( \mathcal{H}^q  M^\bullet) \Rightarrow \mathcal{H}^{p+q} \gr_i^F \dr(  M^\bullet ).
\end{align*}
We have shown $E_2^{p,q} = 0$ for $q \geq c$ and it also vanishes for $p > 0$. This implies $\mathcal{H}^{p+q} \gr_i^F \dr(  M^\bullet ) = 0$ for $p+q \geq c$. The result follows.
\end{proof}

\subsection{Applications}
We now apply our machinery to $\mathbb{D} (\mathbb{Q}_X^H[n]) \in D^b(\MHM(X))$ and related complexes. As a warm-up, this yields a direct proof for the equivalence between the topological (\ref{eqn lcdef top}) and holomorphic (\ref{intro lcdef depth}) characterizations of $\lcdef(X)$.

\begin{theorem}\label{thm lcdef depth body}
$\RHD_\mathbb{Q}(X) = \min_{0 \leq p \leq n} \{ \depth \underline{\Omega}_X^p + p \}$
\end{theorem}

\begin{proof}
By \ref{eqn depth du bois}, we have
\begin{align*}
    \min_{0 \leq p \leq n} \{ \depth \underline{\Omega}_X^p + p \} = 
    - \max \{ j : \mathcal{H}^j \gr_p^F \dr( \mathbb{D}\mathbb{Q}_X^H ) \neq 0 \text{ for some } p \}.
\end{align*}
On the other hand, dualizing the definition of $\RHD_\mathbb{Q}(X)$ yields
\begin{align*}
    \RHD_\mathbb{Q}(X) = - \max\{ j : \mathcal{H}^j( \mathbb{D}\mathbb{Q}_X^H) \neq 0 \}.
\end{align*}
These two quantities are equal by Lemma \ref{lemma hodge-cohom 1}, taking $M^\bullet = \mathbb{D}\mathbb{Q}_X^H$ and $m \gg 0$.
\end{proof}

We prove our main theorem regarding the depth of Du Bois complexes.

\begin{theorem}[= Theorem \ref{intro thm main}]\label{thm main}
Fix integers $k$ and $m$. If
\begin{align*}
    \depth \underline{\Omega}_X^p + p \geq k \hspace{.2cm} \text{ for all } p \leq m,
\end{align*}
then the same inequality holds for $p \geq k-m-2$.
\end{theorem}

The theorem would be a direct consequence of Proposition \ref{prop main} if $\mathbb{D} (\mathbb{Q}_X^H[n])$ had weights $\geq (-n+1)$. This is false, but $\mathbb{D} (\mathcal{K}_X^\bullet)$ has weights $\geq (-n+1)$ by Lemma \ref{lemma wt n-1}. This will be the heart of the matter in the proof below.

\begin{proof}
As $\depth \underline{\Omega}_X^0 \leq n$, we either have $k \leq n$, or else $m \leq -1$. 

Our hypothesis says precisely that
\begin{align*}
    \mathcal{H}^j \gr_p^F \dr( \mathbb{D} (\mathbb{Q}_X^H[n]) ) = 0 \hspace{.5cm} \text{ for all } j \geq -k+1+n \text{ and } p \leq m.
\end{align*}
Let us show the same vanishing holds with $\mathbb{D} \mathcal{K}_X^\bullet$ in place of $\mathbb{D} \mathbb{Q}_X^H[n]$.

Consider the dual of the distinguished triangle \ref{eqn rhm triangle} 
\begin{align*}
    \ic_X^H(n) = \mathbb{D}\ic_X^H \rightarrow \mathbb{D} (\mathbb{Q}_X^H[n]) \rightarrow \mathbb{D} (\mathcal{K}_X^\bullet) \xrightarrow[]{+1},
\end{align*}
together with its cohomology long exact sequence
\begin{align}\label{eqn:ic_les}
\ldots \rightarrow \mathcal{H}^j \gr_{p-n}^F \dr( \ic_X^H )
\rightarrow \mathcal{H}^j \gr_p^F \dr \bigl(\mathbb{D}(\mathbb{Q}_X^H[n])\bigr)
\rightarrow \mathcal{H}^j \gr_p^F \dr\bigl(\mathbb{D}\mathcal{K}_X^\bullet\bigr)
\\
\notag \rightarrow \mathcal{H}^{j+1} \gr_{p-n}^F \dr(\ic_X^H) \rightarrow \ldots
\end{align}
Either $j > 0$ or $p < 0$, so Lemma \ref{lem ic vanishing} implies the last term vanishes. Hence, 
\begin{align*}
    \mathcal{H}^j \gr_p^F \dr( \mathbb{D} \mathcal{K}_X^\bullet ) = 0 \hspace{.5cm} \text{ for all } j \geq -k+1+n \text{ and } p \leq m.
\end{align*}
As $\mathbb{D} \mathcal{K}_X^\bullet$ has weights $\geq (-n+1)$ by Lemma \ref{lemma wt n-1}, we may apply Proposition \ref{prop main} to deduce
\begin{align*}
    \mathcal{H}^j \gr_p^F \dr(\mathbb{D} \mathcal{K}_X^\bullet ) = 0 \hspace{.5cm} \text{ for all } j \geq -k+1+n \text{ and } p \geq k-m-2.
\end{align*}
It now suffices to show the same vanishing holds with $\mathbb{D} (\mathbb{Q}_X^H[n])$ in place of $\mathbb{D} \mathcal{K}_X^\bullet$. In the case $j > 0$, the desired vanishing follows from \ref{eqn:ic_les} and Lemma \ref{lem ic vanishing}. Otherwise, $k \geq n+1$ and hence $m \leq -1$. It follows that $p \geq k-m-2 \geq n$. If some inequality is strict, then $p > n$, and the vanishing follows from \ref{eqn:ic_les} and Lemma \ref{lem ic vanishing}. 

The final case to consider is $j = 0$ and $p = n$, for which we have $\mathcal{H}^0 \gr_{n}^F \dr ( \mathbb{D}(\mathbb{Q}_X^H[n])) = \mathcal{H}^0 \mathbb{D} \underline{\Omega}_X^n$. This vanishes because $\depth \underline{\Omega}_X^n \geq 1$.
\end{proof}

We obtain the following immediate consequence of Theorem \ref{intro thm main}.

\begin{corollary}[= Corollary \ref{intro cor answering conj}]
If $\depth \underline{\Omega}_X^0 \geq j+2$, then $\depth \underline{\Omega}_X^j \geq 2$.
\end{corollary}

This implies Conjecture \ref{intro conj}, because $\depth \underline{\Omega}_X^0 \geq \depth \mathscr{O}_X$ by \ref{eqn depth inequ}.

Combining Theorem \ref{intro thm main} with Theorem \ref{thm lcdef depth body}, we obtain a bound on which $\depth \underline{\Omega}_X^p$ are actually needed to compute $\lcdef(X)$. 

\begin{corollary}[= Theorem \ref{intro thm easy}]\label{cor easy}
Fix an integer $k$. Then we have $\lcdef (X) \leq n-k$ if and only if 
$\depth \underline{\Omega}_X^p + p \geq k$ for all $p \leq \lceil (k-3)/2 \rceil$.
\end{corollary}

\begin{proof}
The only if implication is clear by Theorem \ref{thm lcdef depth body}. Conversely, suppose that $\depth \underline{\Omega}_X^p + p \geq k$ for $p \leq \lceil (k-3)/2 \rceil$. Theorem \ref{intro thm main} says the same inequality holds for $p \geq \lfloor (k-1)/2 \rfloor$. This means it holds for all $p$, so $\lcdef(X) \leq n-k$ by Theorem \ref{thm lcdef depth body}.
\end{proof}

This takes a particularly simple form for $k \leq 3$.

\begin{corollary}\label{cor dao takagi}
If $k \leq 3$, then $\depth \underline{\Omega}_X^0 \geq k$ if and only if $\lcdef(X) \leq n-k$.
\end{corollary}

\begin{remark}\label{rmk injectivity thm}
We recover the following results of Ogus \cite{ogus} and Dao-Takagi \cite{dao takagi}: for $k \leq 3$, if $\depth \mathscr{O}_X \geq k$, then $\lcdef(X) \leq n - k$. Indeed, this follows directly from Corollary \ref{cor dao takagi}, together with the inequality $\depth \mathscr{O}_X \leq \depth \underline{\Omega}_X^0$ from \ref{eqn depth inequ}.
\end{remark}

The following example demonstrates that Corollary \ref{cor dao takagi} is strictly stronger than the result of Dao-Takagi for depth three.

\begin{example}\label{exmp beat dao takagi}
Let $Y$ be a smooth projective surface with ample line bundle $L$, such that $H^1(Y, \mathscr{O}_Y) = 0$ and $H^1(Y, L) \neq 0$ (for example, a double cover $Y$ of $\mathbb{P}^1 \times \mathbb{P}^1$ branched along a smooth $(2, 6)$-divisor, with $L$ the pullback of $\mathscr{O}(1,1)$).

Take $X = \Cone(Y, L)$. We have $\depth \underline{\Omega}_X^0 = 3$, since $H^1(Y, \mathscr{O}_Y) = 0$ by \cite[Theorem $3.1(2)$]{popa shen}. Nevertheless, since $H^1(Y, L) \neq 0$, \cite[Proposition B.$3$]{kebekus schnell} implies $X$ is not Cohen-Macaulay, i.e. $\depth \mathscr{O}_X < 3$. We need Corollary \ref{cor dao takagi} to conclude $\lcdef(X) = 0$.
\end{example}

To extract more from Theorem \ref{intro thm easy}, let us define an invariant
\begin{align*}
    p(X) = \min \{ p \geq 0 : \depth \underline{\Omega}_X^p + p = n - \lcdef(X) \}.
\end{align*}
Although we will not use this in what follows, one can show that $p(X)$ is the smallest index $p$ for which $\gr_{p}^F \mathcal{H}^{\lcdef(X)}\mathbb{D} (\mathbb{Q}_X^H[n])$ is nonzero, using Lemma \ref{lemma hodge-cohom 1}. Moreover, if $\lcdef(X) > 0$ and $X$ satisfies condition $D_m$, then Lemma \ref{lem condtion Dm} says $\depth \underline{\Omega}_X^p+p \geq n$ for $p \leq m$, hence $p(X) \geq m+1$.

\begin{corollary}\label{cor lcdef vs q}
We have
\begin{align*}
    \lcdef(X) \leq n - 2p(X) - 1.
\end{align*}    
In particular, if $X$ satisfies condition $D_m$, then  $\lcdef(X) \leq \max\{ n-2m-3, 0 \}$.
\end{corollary}

The latter statement was proved in \cite[Corollary $7.5$]{popa park lefschetz}.

\begin{proof}
By the definition of $p(X)$,
\begin{align*}
    \depth \underline{\Omega}_X^p + p \geq n-\lcdef(X) + 1 \hspace{.5cm} \text{ for all } p \leq p(X)-1.
\end{align*}
Theorem \ref{intro thm easy} implies $p(X) - 1 < ((n-\lcdef(X)+1)-3)/2$. We obtain the result by rearranging terms.
\end{proof}

The following example shows the first inequality in Corollary \ref{cor lcdef vs q} is strictly stronger than the second, where we take $m$ to be the largest value for which $X$ satisfies condition $D_m$.

\begin{example}\label{exmp diamond}
Let $Y$ be a smooth projective $5$-fold, whose Hodge diamond has the following shape:
\[
\begin{array}{ccccccccccc}
&&&&& 1 &&&&&\\
&&&& 0 && 0 &&&&\\
&&& 0 && 1 && 0 &&&\\
&& 0 && 0 && 0 && 0 &&\\
& 0 && 0 && >1 && 0 && 0 &\\
0 && >0 && * && * && >0 && 0\\
& 0 && 0 && >1 && 0 && 0 &\\
&& 0 && 0 && 0 && 0 &&\\
&&& 0 && 1 && 0 &&&\\
&&&& 0 && 0 &&&&\\
&&&&& 1 &&&&&
\end{array}
\]
Such a variety exists by \cite[Theorem $5$]{construction}. Embed $Y$ in projective space and let $X$ be the projective cone over $Y$.

Then $\lcdef(X) = 1$ by \cite[Theorem A]{popa shen}, as $Y$ first fails to have the same cohomology as $\mathbb{P}^5$ in degree $\dim Y-1$. Nevertheless, $\depth \underline{\Omega}_X^0 = 6$ and $\depth \underline{\Omega}_X^1 = 5$ by \cite[Remark $4.5$]{popa shen}, because this failure only occurs at Hodge degree $2$. From this, we conclude $p(X) = 2$.

Nevertheless, $X$ fails to satisfy condition $D_1$ by \cite[Theorem $7.1$]{popa park lefschetz}. Indeed, the Hodge diamond of $X$ does not satisfy Hodge symmetry on its two outermost layers, due to the nonvanishing of $H^{1,4}(Y)$. It does satisfy condition $D_0$, so we take $m=0$.

We conclude
\begin{align*}
    \lcdef(X) = 1 = n-2p(X)-1 < n - 2m-3 = 3.
\end{align*}
\end{example}

With stronger assumptions, we can use Corollary \ref{cor lcdef vs q} to determine $\lcdef(X)$.

\begin{corollary}[= Theorem \ref{intro dt thm}]\label{cor lcdef = body}
Suppose $X$ satisfies $\depth \underline{\Omega}_X^p +p \geq n$ for all $p \leq m-1$. If $\depth \underline{\Omega}_X^m \leq m+3$, then
\begin{align*}
    \lcdef(X) = \max \{ 0, n-\depth \underline{\Omega}_X^m -m \}.
\end{align*}
The first assumption holds if $X$ satisfies condition $D_{m-1}$.
\end{corollary}

\begin{proof}
Certainly, $\max \{ 0, n-\depth \underline{\Omega}_X^m -m \} \leq \lcdef(X)$. Suppose the inequality is strict. This says $p(X) \geq m+1$, so Corollary \ref{cor lcdef vs q} implies
\begin{align*}
    \lcdef(X) \leq n-2m-3 \leq n - \depth \underline{\Omega}_X^m - m.
\end{align*}
This is contrary to our assumption.
\end{proof}

We also obtain another proof of a recent theorem of Hiatt \cite[Theorem $0.4$]{hiatt}.

\begin{corollary}
Suppose $X$ has pre-$m$-rational singularities for some $m \leq n-2$. Then
\begin{align*}
    \depth \underline{\Omega}_X^p \geq m+2 \hspace{.5cm} \text{ for all } 0 \leq p \leq n
\end{align*}
\end{corollary}

\begin{proof}
Suppose $p \leq n-m-2$. Since $X$ satisfies condition $D_m$, Lemma \ref{lem condtion Dm} implies
\begin{align*}
    \depth \underline{\Omega}_X^q + q \geq n \geq p+m+2 \hspace{.5cm} \text{ for all } q \leq m.
\end{align*}
By Theorem \ref{intro thm main}, we obtain
\begin{align*}
    \depth \underline{\Omega}_X^p \geq m+2.
\end{align*}

If $p \geq n-m$, then $(\mathbb{D} \underline{\Omega}_X^p)[-n]$ is a sheaf in degree zero because $X$ has pre-$m$-rational singularities. This says $\depth \underline{\Omega}_X^p = n$.

For the final case $p = n-m-1$, $\underline{\Omega}_X^{n-m-1} = I\underline{\Omega}_X^{n-m-1}$ by \cite[Proposition $7.4$]{popa park lefschetz}. Since $\depth I \underline{\Omega}_X^{n-m-1} \geq m+1$, it remains to show the vanishing of
\begin{align*}
    \mathcal{H}^{-m-1} \mathbb{D} I \underline{\Omega}_X^{n-m-1} = \mathcal{H}^{n-m-1} I \underline{\Omega}_X^{m+1},
\end{align*}
where the equality holds because $\ic_X^H$ is self-dual, up to twist. Applying $\gr_{-p}^F \dr$ to the triangle \ref{eqn rhm triangle} and taking cohomology, we obtain a long exact sequence
\begin{align*}
    \ldots \rightarrow \mathcal{H}^{n-m-1} \underline{\Omega}_X^{m+1} \rightarrow \mathcal{H}^{n-m-1} I \underline{\Omega}_X^{m+1} \rightarrow \mathcal{H}^1 \gr_{-m-1}^F \dr( \mathcal{K}_X^\bullet) \rightarrow \ldots
\end{align*}
Since $X$ is pre-$m$-Du Bois, we have $\mathcal{H}^{n-m-1} \underline{\Omega}_X^{m+1} = 0$ by \cite[Proposition C]{psv}. Moreover, $\mathcal{K}_X^\bullet$ has cohomology in degrees $\leq 0$ and $\gr_{-m-1}^F \dr(-)$ is concentrated in non-positive degrees, so $\mathcal{H}^1 \gr_{-m-1}^F \dr( \mathcal{K}_X^\bullet) = 0$. The result follows.
\end{proof}

\section{Perversity Defect}\label{section final}

\subsection{Perversity Defect}
This final section is motivated by a comparison of $\lcdef(X)$ with $n - \depth \mathscr{O}_X$. Ogus \cite{ogus} showed $\depth \mathscr{O}_X \geq 2$ implies $\lcdef(X) \leq n-2$, and Dao-Takagi \cite{dao takagi} further proved $\depth \mathscr{O}_X \geq 3$ implies $\lcdef(X) \leq n-3$. The pattern fails beyond this; see \cite[Example $2.11$]{dao takagi} or \cite{kim venkatesh} for examples where $X$ is a toric fourfold. Nevertheless, Dao-Takagi prove that if $\depth \mathscr{O}_X \geq 4$, then $\lcdef(X) \leq n-4$ if and only if $X$ has torsion local analytic Picard groups. 

We aim to generalize this picture by asking what is needed to ensure $\lcdef(X) \leq n-k$, provided we only know $\depth \underline{\Omega}_X^p + p \geq k$ for $p \leq m$ for some fixed value of $m$. Our answer is phrased in terms of the perversity defects of Definition \ref{defn pdef}.

\begin{theorem}[= Theorem \ref{intro thm pdef}]\label{thm pdef}
Let $k$ and $m$ be integers. We have $\lcdef(X) \leq n-k$ if and only if $\depth \underline{\Omega}_X^p + p \geq k$ for all $p \leq m$ and $\pdef(X, k-2m-4) \leq n-k$.
\end{theorem}

This will follow from a more precise result, relating the condition $\depth \underline{\Omega}_X^p + p \geq k$ for $p \leq m$ to the cohomology of small punctured neighborhoods.

\begin{proposition}\label{prop depth link}
Let $k$ and $m$ be integers. Suppose $\depth \underline{\Omega}_X^p + p \geq k$ for all $p \leq m$. If $x \in X$ lies in a Whitney stratum of dimension $d_S \geq k-2m-3$, then
\begin{align*}
    \Tilde{H}^i(U \setminus \{ x \}, \mathbb{Q}) = 0 \hspace{.5cm} \text{ for all } i \leq k - 2 + d_S,
\end{align*}
where $U$ is a small Stein analytic neighborhood of $x$.
\end{proposition}

\begin{proof}
Fix a Whitney stratification on $X$. Let $\iota: Y \rightarrow X$ be the inclusion of a complete intersection of $k-2m-3$ general hyperplanes. It induces a stratification on $Y$, such that each stratum $S$ of $X$ intersects $Y$ transversally as a union of strata $T$. Moreover, if $y$ lies in a stratum $T$ of $Y$, there exists a small neighborhood $U$ of $y$ in $X$ of the form $V \times B$, where $V \subset T$ is a small neighborhood and $B$ is a contractible ball. In particular, $U \setminus (U \cap S)$ is homotopy equivalent to $V \setminus (V \cap T)$. Combining this with \cite[Lemma $15.10$]{popa park lefschetz}, we have
\begin{align}\label{eqn hypersurface strata}
    \Tilde{H}^i(U \setminus \{ y \}, \mathbb{Q}) &\simeq \Tilde{H}^{i-2d_S}(U \setminus (U \cap S), \mathbb{Q}) \\ &\simeq \Tilde{H}^{i-2d_S}(V \setminus (V \cap T), \mathbb{Q}) \notag\\
    &\simeq \Tilde{H}^{i-2d_S+2d_T}(V \setminus \{ y \}, \mathbb{Q}) \notag.
\end{align}
Here, $d_S - d_T = \codim_Y X = k-2m-3$. We will continue using this notation throughout the proof.

Lemma \ref{lem depth hyperplane} implies $\depth \underline{\Omega}_Y^p + p \geq 2m+3$ for all $p \leq m$. By Theorem \ref{intro thm easy}, we deduce $\lcdef(Y) \leq \dim Y-2m-3$. In other words, by Lemma \ref{lemma link},
\begin{align*}
    \Tilde{H}^i(V \setminus \{ y \}, \mathbb{Q}) = 0 \hspace{.5cm} \text{ for all } i \leq 2m+1+d_T \text{ and all } y \in Y.
\end{align*}
By \ref{eqn hypersurface strata}, this is equivalent to
\begin{align*}
    \Tilde{H}^i(U \setminus \{ y \}, \mathbb{Q}) = 0 \hspace{.5cm} \text{ for all } i \leq k-2 + d_S \text{ and all } y \in Y.
\end{align*}
The result follows, because the cohomology of a punctured neighborhood depends only on the stratum of the removed point, and the strata $S$ of $X$ with $S \cap Y \neq \emptyset$ are precisely those of dimension $d_S \geq k-2m-3$.
\end{proof}

We required the following lemma, showing the condition $\depth \underline{\Omega}_X^p + p \geq k$ for $p \leq m$ is well-behaved under general hyperplane sections.

\begin{lemma}\label{lem depth hyperplane}
Let $k$ and $m$ be integers, and let $\iota: Y \rightarrow X$ be a general hyperplane section. If $\depth \underline{\Omega}_X^p + p \geq k$ for all $p \leq m$, then $\depth \underline{\Omega}_Y^p + p \geq k-1$ for all $p \leq m$.
\end{lemma}

\begin{proof}
Our assumption says precisely that 
\begin{align*}
     \mathcal{H}^j \gr_p^F \dr( \mathbb{D}(\mathbb{Q}_X^H[n])) = 0 \hspace{.2cm} \text{ for all } j \geq -k+1+n \text{ and } p \leq m.
\end{align*}
Applying Lemma \ref{lemma hodge-cohom 1} and dualizing, we see this is equivalent to
\begin{align*}
     \mathcal{H}^{-j} \gr_{-p}^F (\mathbb{Q}_X^H[n]) = 0 \hspace{.2cm} \text{ for all } j \geq -k+1+n \text{ and } p \leq m.
\end{align*}
Pulling back by $\iota^*$ and using the fact that $Y$ is general, \cite[Lemma $2.25$]{saito 90} and \cite[Theorem $9.3$]{schnell}, together with the identity $\iota^* \mathbb{Q}_X^H = \mathbb{Q}_Y^H$ imply
\begin{align*}
     \mathcal{H}^{-j} \gr_{-p}^F(\mathbb{Q}_Y^H[n-1]) = 0 \hspace{.2cm} \text{ for all } j \geq -k+1+n \text{ and } p \leq m.
\end{align*}
Applying the same argument in reverse, we conclude the proof.
\end{proof}

\begin{proof}[Proof of Theorem \ref{thm pdef}]
The only if direction follows from Theorem \ref{thm lcdef depth body} and the inequalities \ref{eqn inequality chain}. For the converse, we need to show
\begin{align*}
    \Tilde{H}^i(U \setminus \{ x\}, \mathbb{Q}) = 0 \hspace{.5cm} \text{ for all } i \leq k-2+d_S \text{ and all } x \in X.
\end{align*}
For strata of dimension $d_S \geq k-2m-3$, this follows from the depth hypothesis due to Proposition \ref{prop depth link}. Otherwise, $d_S \leq k-2m-4$, and the vanishing follows from the condition $\pdef(X, k-2m-4) \leq n-k$, interpreted in terms of cohomology groups of punctured neighborhoods, as in Lemma \ref{lemma link}.
\end{proof}

Theorem \ref{intro thm pdef} has several interesting corollaries. For example, if $m = \lceil (k-3)/2 \rceil$, the perversity defect hypothesis is vacuous, so the statement is equivalent to Theorem \ref{intro thm easy}. In the next case, taking $k=2m+4$, we obtain:

\begin{corollary}\label{cor cdef}
Let $k$ be an even integer. Then $\lcdef(X) \leq n-k$ if and only if $\depth \underline{\Omega}_X^p + p \geq k$ for all $p \leq (k-4)/2$ and $\cdef(X) \leq n-k$.
\end{corollary}

\begin{example}[Dao-Takagi Theorem for depth four]\label{example dt}
We recover the motivating result of Dao-Takagi \cite{dao takagi}: if $\depth \mathscr{O}_X \geq 4$, then $\lcdef(X) \leq n-4$ if and only if $X$ has torsion local analytic Picard groups.

Indeed, suppose $\depth \mathscr{O}_X \geq 4$. Then $\depth \underline{\Omega}_X^0 \geq 4$ and $\cdef(X) \leq \lcdef(X) \leq n-3$ by Corollary \ref{cor dao takagi}. Using Corollary \ref{cor cdef} with $k=4$ and interpreting $\cdef(X)$ in terms of cohomology groups of punctured neighborhoods as in Lemma \ref{lemma link}, it suffices to show that $H^{2}(U \setminus \{ x\}, \mathbb{Q})$ is isomorphic to $\Pic^{\anloc}(X,x)_\mathbb{Q}$ for all points $x \in X$.

The exponential sequence for $U \setminus \{ x\}$ gives rise to a long exact sequence
\begin{align*}
    \ldots H^1(U \setminus \{ x\}, \mathscr{O})& \rightarrow H^1(U \setminus \{ x\}, \mathscr{O}^*) \\& \rightarrow H^2(U \setminus \{ x\}, \mathbb{Z}) \rightarrow H^2(U \setminus \{ x\}, \mathscr{O}) \rightarrow \ldots
\end{align*}
The depth assumption $\depth \mathscr{O}_X \geq 4$ implies the two outermost terms vanish \cite[Theorem $1.14$]{siu}, so $H^2(U \setminus \{ x\}, \mathbb{Z}) \simeq H^1(U \setminus \{ x \}, \mathscr{O}^*)$. The latter is isomorphic to $\Pic^{\anloc}(X,x)$ by \cite[Section $1$]{kollar}.
\end{example}

We can also take $m = 0$ in Theorem \ref{intro thm pdef} to deduce:

\begin{corollary}\label{cor pdef}
If $k$ is an integer, then $\lcdef(X) \leq n-k$ if and only if $\depth \underline{\Omega}_X^0 \geq k$ and $\pdef(X,k-4) \leq n-k$.
\end{corollary}

\begin{example}[Park-Popa Theorem for depth five]\label{exmp lcdef 5}
Let us discuss how this result in the case $k=5$ implies \cite[Corollary G]{popa park lefschetz}: if $n \geq 5$ and $X$ has rational locally analytically $\mathbb{Q}$-factorial singularities, then $\lcdef(X) \leq n-5$ if and only if $X$ has torsion local analytic gerbe groups. We make crucial use of the fact that rational locally analytically $\mathbb{Q}$-factorial singularities are preserved by general hyperplane sections \cite[Proposition $12.1$]{popa park lefschetz}.

Under these hypotheses, we have $\depth \underline{\Omega}_X^0 = \depth \mathscr{O}_X = n$. Moreover, $X$ has torsion local analytic Picard groups, so $\lcdef(X) \leq n-4$ by Example \ref{example dt}.

As in Example \ref{example dt}, the exponential long exact sequence shows $X$ has local analytic gerbe groups if and only if
\begin{align*}
    H^3(U \setminus \{ x\}, \mathbb{Q}) = 0 \hspace{.5cm} \text{ for all } x \in X.
\end{align*}
On the other hand, by Lemma \ref{lemma link}, we have $\pdef(X,1) \leq n-5$ if and only if
\begin{align}
    \Tilde{H}^i(U \setminus \{ x\}, \mathbb{Q}) = 0 \hspace{.5cm} \text{ for all } i \leq 3+\min \{ d_S, 1\} \text{ and all } x \in X.
\end{align}
We aim to show the equivalence of these conditions. Since $\lcdef(X) \leq n-4$, the only vanishing left to show is that of $H^4(U \setminus \{ x \}, \mathbb{Q})$, when $d_S = 1$. 

Let $\iota: Y \rightarrow X$ be a general hyperplane section through $x$. Then $H^4(U \setminus \{ x \}, \mathbb{Q}) \simeq H^2(V \setminus \{ x \}, \mathbb{Q})$ by \ref{eqn hypersurface strata}, where $V \subset Y$ is a small neighborhood of $x$. The latter cohomology group is exactly $\Pic^{\anloc}(Y,x)_\mathbb{Q}$ by the argument in Example \ref{example dt}. But $Y$ has rational locally analytically $\mathbb{Q}$-factorial singularities, so its local analytic Picard groups are torsion. Hence $\Pic^{\anloc}(Y,x)_\mathbb{Q} = 0$.
\end{example}


\begin{thebibliography}{9}

\bibitem[AH24]{weaklyratl}
Donu Arapura and Scott Hiatt. Differential Forms and Hodge Structures on Singular Varieties. \emph{arXiv preprint arXiv:2410.21007}.

\bibitem[BBD82]{bbd}
Alexander A. Beilinson, Joseph Bernstein, and Pierre Deligne. Faisceaux pervers.  In \emph{Analysis and topology on singular spaces, I (Luminy, 1981)}, volume 100 of \emph{Astérisque}, pages 5-171. Soc. Math. Franace, Paris, (1982).

\bibitem[BBL{$^+$}25]{bbl}
Bhargav Bhatt, Manuel Blickle, Gennady Lyubeznik, Anurag K. Singh, and Wenliang Zhang. Applications of perverse sheaves in commutative algebra. \emph{J. Reine Angew. Math.} {\bf 825} (2025), 81--138.

\bibitem[DT16]{dao takagi}
Hailong Dao and Shunsuke Takagi. On the relationship between depth and cohomological dimension \emph{Compos. Math.} {\bf 152} (2016), no. 4, 876--888.

\bibitem[DOR25]{dod}
Bradley Dirks, Sebasti{\'a}n Olano, and Debatiya Raychadhury. A Hodge theoretic generalization of Q-homology manifolds. \emph{arXiv preprint arXiv:2501.14065}.

\bibitem[DB81]{du bois}
Philippe Du Bois. Complexe de de Rham filtr{\'e} d'une vari{\'e}t{\'e} singuli{\`e}re. \emph{Bull. Soc. Math. France} {\bf 109} (1981), no. 1, 41--81.

\bibitem[Hi25]{hiatt}
Scott Hiatt. Differential forms on varieties with pre-$k$-rational singularities. \emph{arXiv preprint arXiv:2509.04382}.

\bibitem[HJ13]{h-topology}
Annette Huber and Clemens J{\"o}rder. Differential forms in the h-topology. \emph{Algebr. Geom.} {\bf 1} (2014), no. 4, 449--478.

\bibitem[KS06]{kashiwara}
Masaki Kashiwara and Pierre Schapira. Categories and Sheaves. \emph{Grundlehren der Mathematischen Wissenschaften}, {\bf 332} (2006). Springer-Verlag, Berlin (2006).

\bibitem[KS21]{kebekus schnell}
Stefan Kebekus and Christian Schnell. Extending holomorphic forms from the regular locus of a complex space to a resolution of singularities. \emph{J. Amer. Math. Soc.} {\bf 34} (2021), no. 2, 315--368.

\bibitem[KV25]{kim venkatesh}
Hyunsuk Kim and Sridhar Venkatesh. Lefschetz morphisms on singular cohomology and local cohomological dimension of toric varieties. \emph{arXiv preprint arXiv:2506.03026}.

\bibitem[KS16]{kovacs schwede}
S{\'a}ndor Kov{\'a}cs and Karl Schwede. Du Bois singularities deform. \emph{Minimal Models and extremal rays (Kyoto, 2011), Adv. Stud. in Pure Math., Math. Soc. Japan} {\bf 70} (2016), 49--65.

\bibitem[Ko12]{kollar}
J{\'a}nos Koll{\'a}r. Grothendieck--Lefschetz type theorems for the local Picard group. \emph{J. Ramanujan Math. Soc.} {\bf 28A} (2013), 267--285.

\bibitem[MP22]{local cohomology}
Mircea Musta{\c{t}}{\u{a}} and Mihnea Popa. Hodge filtration on local cohomology, Du Bois complex and local cohomological dimension. \emph{Forum Math. Pi} {\bf 10} (2022), e22.

\bibitem[Og73]{ogus}
Arthur Ogus. Local cohomological dimension of algebraic varieties. \emph{Ann. of Math.} {\bf 98} (1973), no. 2, 327--365.

\bibitem[Pa23]{park}
Sung Gi Park. Du Bois complex and extension of forms beyond rational singularities. \emph{arXiv preprint arXiv:2311.15159}.

\bibitem[PP25]{popa park lefschetz}
Sung Gi Park and Mihnea Popa. Hodge symmetry and Lefschetz theorems for singular varieties. \emph{arXiv preprint arXiv:2410.15638}.

\bibitem[PS20]{construction}
Matthias Paulsen and Stefan Schreieder. The construction problem for Hodge numbers modulo an integer. \emph{Algebra Number Theory} {\bf 13} (2019), no. 10, 2427--2434.

\bibitem[PS24]{popa shen}
Mihnea Popa and Wanchun Shen. Du Bois complexes of cones over singular varieties, local cohomological dimension, and K-groups. \emph{Rev. Roumaine Math. Pures Appl., L. B\u{a}descu memorial issue} {\bf 70} (2025), no. 1--2, 133--155.

\bibitem[PSV24]{psv}
Mihnea Popa, Wanchun Shen, and Anh Duc Vo. Injectivity and vanishing for the Du Bois complexes of isolated singularities. \emph{arXiv preprint arXiv:2409.18019, to appear in Algebra and Number Theory} (2024).

\bibitem[RSW21]{saito lcdef}
Thomas Reichelt, Morihiko Saito, and Uli Walther. Topological calculation of local cohomological dimension. \emph{J. Singul.} {\bf 26} (2023), 13--22.

\bibitem[Sa87]{saito intro}
Morihiko Saito. Introduction to mixed Hodge modules. \emph{Astérisque} {\bf 179--180} (1987), 145--162.

\bibitem[Sa88]{saito 88}
Morihiko Saito. Modules de Hodge polarisables. \emph{Publ. Res. Inst. Math. Sci.} {\bf 24} (1988), no. 6, 849--995.

\bibitem[Sa90]{saito 90}
Morihiko Saito. Mixed Hodge modules. \emph{Publ. Res. Inst. Math. Sci.} {\bf 26} (1990), no. 2, 221--333.

\bibitem[Sa00]{saito hodge complexes}
Morihiko Saito. Mixed Hodge complexes on algebraic varieties. \emph{Math. Ann.} {\bf 316} (2000), no. 2, 283--331.

\bibitem[Sc14]{schnell}
Christian Schnell. On Saito's vanishing theorem. \emph{Math. Res. Lett.} {\bf 23} (2016), no. 2, 499--527.

\bibitem[ST06]{siu}
Yum-Tong Siu and G{\"u}nther Trautmann. Gap-sheaves and extension of coherent analytic subsheaves, volume Vol. 172 of \emph{Lecture Notes in Mathematics}. Springer-Verlag, Berlin-New York, (1971).

\bibitem[SVV23]{shen venkatesh vo}
Wanchun Shen, Sridhar Venkatesh, and Anh Duc Vo. On $k$-Du Bois and $k$-rational singularities. \emph{arXiv preprint arXiv:2306.03977, to appear in Ann. Inst. Fourier}, (2023).


\end{thebibliography}
\end{document}